\documentclass[12pt,final]{article}
\usepackage{amsmath,amsthm,amsfonts,amssymb,graphicx,hyperref,enumerate,psfrag,textgreek}
\usepackage{fullpage}

\usepackage[utf8]{inputenc} 

\newtheorem{lemma}{Lemma}

\newtheorem{remark}{Remark}
\newtheorem{proposition}{Proposition}
\newtheorem{theorem}{Theorem}

\newtheorem{conjecture}[lemma]{Conjecture}

\newtheorem{corollary}{Corollary}

\newcommand{\vertiii}[1]{{\left\vert\kern-0.25ex\left\vert\kern-0.25ex\left\vert #1 
    \right\vert\kern-0.25ex\right\vert\kern-0.25ex\right\vert}}
\newcommand{\EE}{{\mathbb{E}}}

\newcommand{\PP}{\mathbb{P}}
\newcommand{\Z}{\mathbb {Z}}

\newcommand{\R}{\mathbb {R}}

\newcommand{\cI}{\mathcal {I}}
\newcommand{\cJ}{\mathcal {J}}

\newcommand{\cE}{\mathcal {E}}

\newcommand{\cL}{\mathcal {L}}
\newcommand{\cP}{\mathcal {P}}

\newcommand{\GG}{{\mathbb G}}

\newcommand{\cK}{\mathcal{K}}

\newcommand{\supp}{ \mathrm{supp}}

\newcommand{\dtv}{d_{\textsc{tv}}}

\newcommand{\mix}{t_{\textsc{mix}}}

\newcommand{\IP}{{\textsc{ip}}}
\newcommand{\RW}{{\textsc{rw}}}
\newcommand{\EX}{{\textsc{ex-}k}}
\newcommand{\fS}{\mathfrak{S}}
\newcommand{\tmix}{t_{\textsc{mix}}}
\newcommand{\trel}{t_{\textsc{rel}}}
\newcommand{\comp}{\chi}

\newcommand{\cay}{\textrm{Cay}}
\title{The interchange process on high-dimensional products}
\author{Jonathan Hermon, Justin Salez}

\begin{document}

\maketitle
\begin{abstract}
We resolve a long-standing conjecture of Wilson (2004), reiterated by Oliveira (2016), asserting that the mixing time of the Interchange Process with unit edge rates on the $n$-dimensional hypercube is of order $n$. This follows from a sharp inequality established at the level of Dirichlet forms, from which we also deduce that macroscopic cycles emerge in constant time, and that the log-Sobolev constant of the exclusion process is of order $1$. Beyond the hypercube, our results apply to cartesian products of arbitrary graphs of fixed size, shedding light on a broad conjecture of Oliveira (2013). 
\end{abstract}

\vspace{-0.5cm}
\tableofcontents

\newpage
\section{Introduction}
\subsection{Interchange process}
Let $G=(V_G,E_G)$ be a finite undirected connected graph. The \emph{interchange process} (\IP) on $G$ is the continuous-time random walk $(\xi_t)_{t\ge 0}$  on the symmetric group $\fS(V_G)$ with initial condition $\xi_0=\textrm{id}$ and the following Markov generator: for all observables $f\colon \fS(V_G)\to\R$,
\begin{eqnarray}
\cL^{\IP}_G f(\sigma) & := & \sum_{e\in E_G}\left(f(\sigma\tau_e)-f(\sigma)\right),
\end{eqnarray}
 where $\tau_e$ denotes the transposition of the endpoints of the edge $e$. One may think of each vertex as carrying a labelled particle, and of the edges as being equipped with independent unit-rate Poisson clocks: whenever a clock rings, the particles sitting at the endpoints of the corresponding edge simply exchange their positions.  

Since $\cL^{\IP}_G$ is symmetric and irreducible, the law of $\xi_t$ converges to that of a uniform permutation $\xi_\star$ as $t\to\infty$. We shall here be interested in the time-scale on which this convergence occurs, as traditionally measured
by the  total-variation mixing time:
\begin{eqnarray}
\label{tmix}
\tmix^{\IP}(G) & := & \min\left\{t\ge 0\colon \max_{A\subseteq \fS(V_G)}\left|\PP\left(\xi_t\in A\right)-\PP(\xi_\star\in A)\right|\le \frac 1e\right\}.
\end{eqnarray}
Understanding the relation between this fundamental quantity and the geometry of $G$ is a challenging problem, to which a remarkable variety of tools have been  applied: representation theory \cite{DS,MR964069}, couplings \cite{levin, Wilson,aldous,olive},  eigenvectors \cite{Wilson,MR3069380}, functional inequalities \cite{LY,Yau,Caputo}, comparison methods \cite{comparison,AK}, etc. 
 The question is of course particularly meaningful when the number of states becomes large, and one is thus naturally led to study  asymptotics  along various growing sequences of graphs $(G_n)_{n\ge 1}$. 

The case of the $n-$clique $G_n=\cK_n$  has been extensively studied under the name \emph{random transposition shuffle}. In particular,  Diaconis and Shahshahani \cite{DS} proved that 
\begin{eqnarray}
\tmix^{\IP}(\cK_n) & = & \frac{\log n}{n}\left(1+o(1)\right).
\end{eqnarray}
In fact, this was shown for any precision $\varepsilon\in(0,1)$ instead of $\frac{1}{e}$ in (\ref{tmix}), thereby extablishing the very first instance of what is now called a \emph{cutoff phenomenon} \cite{MR1374011}. Another well-understood case  is  the $n-$path $\cP_n$, for which Lacoin \cite{lacoin}  recently proved cutoff at time  
\begin{eqnarray}
\tmix^{\IP}(\cP_n) & = & \frac{n^2 \log n}{2\pi^2}\left(1+o(1)\right).
\end{eqnarray}

There are, however, many simple graph sequences along which even the order of magnitude of $\tmix^{\IP}(G_n)$ is unknown. An emblematic example (which was the initial motivation for our work) is the boolean hypercube $\Z_2^n$, for which Wilson \cite{Wilson} conjectured in $2004$ that 
\begin{eqnarray}
\label{conj:wilson}
\tmix^{\IP}(\Z_2^n) & \asymp & n.
\end{eqnarray}
This was reiterated as Problem 4.2 of the  AIM workshop \emph{Markov Chains Mixing Times} \cite{Oliveaim}. 
Here and throughout the paper, $\asymp$ and $\lesssim$ denote  equality and inequality up to \underline{universal} positive multiplicative constants.  
The current best estimates  are
\begin{eqnarray}
n \ \lesssim  & \tmix^{\IP}(\Z_2^n) & \lesssim \ n\log n.
\end{eqnarray}
The lower bound is due to Wilson \cite[Section 9.1]{Wilson}, and the upper bound was recently obtained by Alon and Kozma \cite[Corollary 10]{AK} as a special case of a much more general estimate which we will now discuss.

\subsection{The big picture} An important observation about the \IP\ 
 is that the motion of a single particle is itself a Markov process. The generator is the usual graph Laplacian, which acts on functions $f\colon V_G\to\R$ by
\begin{eqnarray}
\label{def:RW}
\cL^{\RW}_G f(x) & := & \sum_{y\colon \{x,y\}\in E_G}\left(f(y)-f(x)\right).
\end{eqnarray}
It is   natural to expect the mixing properties of $\cL^{\IP}_G$ and $\cL^{\RW}_G$ to be intimately related. Indeed, a celebrated conjecture of Aldous, now resolved by Caputo, Liggett and Richthammer \cite{Caputo}, asserts that the \emph{relaxation times} (inverse spectral gaps) of these two operators coincide:
\begin{eqnarray}
\label{Caputo}
\trel^{\IP}(G) & = & \trel^{\RW}(G).
\end{eqnarray}
Recall that $\trel^{\RW}(G)$ classically  controls the mixing time $\tmix^{\RW}(G)$ of the single-particle dynamics (\ref{def:RW}), up to a correction which is only logarithmic in the number of vertices:
\begin{eqnarray}
\tmix^{\RW}(G) & \lesssim &  \trel^{\RW}(G)\log{|V_G|}. 
\end{eqnarray}
Inspired by the identity (\ref{Caputo}), Oliveira \cite{olive} conjectured that  the same control applies to $\tmix^{\IP}(G)$. More precisely, he proposed the following simple-looking but far-reaching estimate, which is sharp in the three very different graph examples mentioned above (see Table \ref{table}). 
\begin{conjecture}[Oliveira \cite{olive}]\label{conj:main} For any connected graph $G$,
\begin{eqnarray}
\label{eq:main}
\tmix^{\IP}(G) & \lesssim &  \trel^{\RW}(G)\log |V_G|. 
\end{eqnarray}
\end{conjecture}

Some partial progress on Conjecture \ref{conj:main} can be found in \cite{Pymar,Pymar2}. It is easy to prove that $t_{\mathrm{rel}}^{\mathrm{RW}} \log |V|$ is comparable up to some universal constants to the mixing time of $|V|$ independent particles \cite[\S2.1]{Pymar} and thus Oliveira's conjecture has the following probabilistic interpretation. It is saying that the mixing time of the interchange process is at most some universal constant multiple of the mixing time of $|V|$ independent particles (in fact, this is how it phrased in \cite{olive}). . We also note that under a mild spectral condition one has that $\tmix^{\IP}(G)  \gtrsim   \trel^{\RW}(G)\log |V_G|$ \cite[Theorem 1.4]{Pymar}, but that in general one can have that  $\tmix^{\IP}(G)$ is of smaller order than  $\trel^{\RW}(G)\log |V_G|$.\footnote{One such example can be obtained by taking $G_n$ to be the graph obtained by attaching a path of some diverging length $L_n$ to a clique of size $n$, where $ \log L_n =o(\log n)$. A similar example is analyzed in \cite{HK}, where it is shown that the order of the total variation mixing time of the interchange process may increase as a result of increasing some of the edge rates by a $1+o(1)$ multiplicative factor, or by adding a small number of edges to the base graph (in a manner that makes the original graph and the new graph quasi-isometric).}

One of the most powerful techniques to bound the mixing time of a complicated Markov chain consists in comparing its {Dirichlet form}  with that of a better understood chain having the same state space and stationary law, see the seminal paper by Diaconis and Saloff-Coste \cite{comparison}. In the case of $\IP$, the {Dirichlet form} is given by
\begin{eqnarray}
\cE^{\IP}_G\left(f\right) & := & \frac{1}{2|V_G|!}\sum_{\sigma\in\fS(V_G)}\sum_{e\in E_G}\left(f(\sigma\tau_e)-f(\sigma)\right)^2,
\end{eqnarray}
 and a natural candidate for the  comparison is  the \emph{mean-field} version $\cE^{\IP}_\cK$, where $\cK$ denotes the complete graph on the same vertex set as $G$. Let us therefore define the \emph{comparison constant} of the $\IP$ on $G$ as  the smallest number $\comp_G^\IP$  such that the inequality
 \begin{eqnarray}
 \label{def:comp}
\cE^{\IP}_{\cK}\left(f\right) & \le & \comp_G^\IP\,\cE^{\IP}_G\left(f\right) 
 \end{eqnarray}
holds for all $f\colon \fS(V_G) \to\R$. This constant is the optimal price to pay in order to systematically transfer quantitative estimates from $\IP$ on $\cK$ to $\IP$ on $G$. In a recent breakthrough, Alon and Kozma \cite[Theorem 1]{AK}  established the following remarkably general estimate.
\begin{theorem}[Alon and Kozma \cite{AK}]\label{th:AK}For any regular connected graph $G$,
\begin{eqnarray}
\comp^{\IP}_G & \lesssim & |V_G|\,\tmix^\RW(G).
\end{eqnarray}
\end{theorem}
In particular, they deduced the following bound on the mixing times.
\begin{corollary}[Alon and Kozma \cite{AK}]For any regular connected graph $G$,
\begin{eqnarray}
\tmix^{\IP}(G) & \lesssim & \tmix^{\RW}(G)\log |V_G|. 
\end{eqnarray}
\end{corollary}

Note that this proves Conjecture \ref{conj:main} along sequences $(G_n)_{n\ge 1}$ satisfying $\tmix^{\RW}(G_n)\asymp \trel^{\RW}(G_n)$. Examples include $\cK_n$, $\cP_n$, or the discrete tori $\Z_n,\Z_n^2,\Z_n^3$, etc. On the other hand, for various other graphs such as the hypercube $\Z_2^n$ or bounded-degree expanders, one has
\begin{eqnarray}
\label{highdim}
\frac{\trel^{\RW}(G_n)}{\tmix^{\RW}(G_n)} & \xrightarrow[n\to\infty]{} & 0,
\end{eqnarray}
and Theorem \ref{th:AK} fails at capturing the conjectured asymptotics. In light of this, the next step towards Conjecture \ref{conj:main} should naturally consist in understanding the mixing properties of the \IP\ on graphs satisfying (\ref{highdim}). This is precisely the program to which the present paper is intended to contribute.

\begin{table}
\begin{center}
\begin{tabular}{ |c|c|c|c|c| }
 \hline
 $G$ & $|V_G|$ & $\trel^\RW(G)$ & $\mix^\RW(G)$ & $\tmix^\IP(G)$ \\  
 \hline
 $\cK_n$ & $n$ & $1/n$ & $1/n$ & $(\log n)/n$ \\ 
 \hline
 $\cP_n$ & $n$ &   $n^2$ & $n^2$ & $n^2\log n$ \\ 
 \hline
 $\Z_2^n$  & $2^n$ & $1$ & $\log n$ & $n$\\ 
 \hline
\end{tabular}
\caption{\label{table}Some classical orders of magnitude.}
\end{center}
\end{table}

\section{Results}
\subsection{Comparison constant and mixing time}
A natural and important class of graphs satisfying (\ref{highdim}) are the ``high-dimensional'' graphs obtained by taking cartesian products of a large number of small graphs. Recall that the \emph{cartesian product} $G=G_1\times\cdots \times G_n$ of $n$ graphs $G_1,\ldots,G_n$ is the graph with vertex set
$
V_{G_1}\times\cdots\times V_{G_n},
$ 
and where the neighbors of a vertex $x=(x_1,\ldots,x_n)$ are obtained by replacing an arbitrary coordinate $x_i$ ($1\le i\le n$) with an arbitrary neighbor of $x_i$ in the graph $G_i$. Note that $G$ is connected as soon as $G_1,\ldots,G_n$ are. We will allow the \emph{dimension} $n$ to grow arbitrarily but will keep the \emph{side-length} fixed, meaning that
\begin{eqnarray}
V_{G_1} \ = \ \ldots \ = \ V_{G_n} \ = \ \{1,\ldots,\ell\}
\end{eqnarray}
for some fixed integer $\ell\ge 2$. 
Two simple examples are the $n-$dimensional torus  
$
\Z_\ell^n  =  \Z_\ell\times \cdots\times \Z_\ell,
$
and the $n-$dimensional \emph{Hamming graph} $\cK_\ell^n=\cK_\ell\times \cdots\times \cK_\ell$. In particular, when $\ell=2$, we recover the hypercube. Our main result is the  determination of the exact order of magnitude of $\comp^{\IP}_G$ on all product graphs of fixed side-length. 
\begin{theorem}[Comparison]\label{th:main}All connected product graphs of side-length $\ell\ge 2$ satisfy
\begin{eqnarray}
\comp^{\IP}_G & \asymp_\ell &  |V_G|,
\end{eqnarray}
where $\asymp_\ell$ means equality up to multiplicative constants that depend only on $\ell$.
\end{theorem}
Our estimate on $\comp^{\IP}_G$ classically yields an upper bound on the mixing time, even in the strong $L^2$ sense (see \cite[Lemma 6]{AK}). Moreover, a standard application of Wilson's method (see \cite[Proposition 1.2]{MR3069380}) yields a matching lower bound. We thus  obtain the following result, which confirms in particular Wilson's long-standing prediction (\ref{conj:wilson}).    
\begin{corollary}[Mixing time]\label{co:mix}All connected product graphs of side-length $\ell\ge 2$ satisfy
\begin{eqnarray}
\label{eq:mix}
\tmix^{\IP}(G) \ \asymp_\ell \ \log |V_G|.
\end{eqnarray}
\end{corollary}
Note that on product graphs, the single-particle dynamics (\ref{def:RW}) updates each coordinate  independently. Consequently, any connected product graph of side-length $\ell\ge 2$ satisfies
\begin{eqnarray}
\label{0regime}
\trel^{\RW}\left(G\right) & \asymp_\ell & 1,\\ 
\label{regime}
\tmix^{\RW}\left(G\right) & \asymp_\ell & \log\log |V_G|.
\end{eqnarray}
(The double logarithm comes from the fact that there are $\log_\ell |V_G|$ coordinates, and that the time it takes to update all of them a constant number of times is logarithmic in the number of coordinates).
Thus, our Corollary \ref{co:mix} resolves Conjecture \ref{conj:main} for all product graphs of fixed side-length, in a regime where Theorem \ref{th:AK} always fails at doing so.  

\begin{remark}[Pre-cutoff] Let us comment on the constants hidden in our results, at least for the Hamming graph $\cK_\ell^n$. Wilson's eigenvector method \cite{Wilson} produces the precise lower bound 
\begin{eqnarray}
\liminf_{n\to\infty}\left\{\frac{\tmix^\IP(\cK_\ell^n)}{\ell^n}\right\} & \ge & c_\ell,
\end{eqnarray}
with the constant $c_\ell>0$ being completely explicit (for example, we have $c_2=\frac{\ln 4}{2}$). On the other hand, our Corollary \ref{co:mix} guarantees that
\begin{eqnarray}
\limsup_{n\to\infty}\left\{\frac{\tmix^\IP(\cK_\ell^n)}{\ell^n}\right\} & \le & C_\ell,
\end{eqnarray}
for some constant $C_\ell<\infty$ that can certainly be made explicit as well, by a careful examination of our proof. However, we did not try to optimize the value of $C_\ell$, nor even to extract its rough dependency in $\ell$, because we believe that our comparison-based approach is inherently too rough to produce sharp constants anyway. Nevertheless, we note that neither Wilson's lower bound $c_\ell$ nor our upper bound $C_\ell$ change if we replace $\frac 1e$ by any other precision $\varepsilon\in(0,1)$ in the definition (\ref{tmix}), thereby establishing what is known as  a  \emph{pre-cutoff}. Improving this to a true \emph{cutoff} (i.e. $C_\ell=c_\ell$) remains a fascinating open problem.
\end{remark}

We would like to close this section with a plausible extension of Theorem \ref{th:main}, inspired by an analogous result that we recently obtained for the \emph{Zero-Range Process} \cite[Corollary 3]{HS}. 
\begin{conjecture}[General comparison]All finite connected graphs satisfy
\begin{eqnarray}
\comp_G^\IP & \lesssim & |V_G|\,\trel^\RW(G).
\end{eqnarray}
\end{conjecture}
Note that a proof of this would immediately imply Conjecture \ref{conj:main}.  

\subsection{Emergence of macroscopic cycles}

One statistics of particular interest is  the cycle structure of the random permutation $\xi_t$, as a function of the time $t$. On the infinite $d-$dimensional lattice $\Z^d$ with $d \ge 3$, a long-standing conjecture of T\'oth \cite{Toth} predicts a phase transition, indicated by the sudden emergence of infinite cycles at some critical time $t=t_c\in(0,\infty)$. This  is related to a major open problem about the so-called \emph{quantum Heisenberg ferromagnet} in statistical mechanics. To the best of our knowledge, the phase transition has only been proved on infinite regular trees \cite{cycleangel,cycleHammond}.

In the case of a large finite graph $G$, the relative lengths of cycles in  a uniform random permutation  asymptotically follow the Poisson-Dirichlet distribution (see, e.g., \cite{cycleSchramm}). In particular, $\xi_t$ is likely to  contain a macroscopic cycle at time $t\ge\tmix^{\IP}(G)$. By analogy with T\'oth's conjecture, one should however expect macroscopic cycles to emerge much before the mixing time. This was established in a  precise sense by Schramm \cite{cycleSchramm} in the {mean-field} case where $G=\cK_n$, see also \cite{cycleberes,cycleBK}. Alas,  results on other finite graphs are quite limited. In \cite{cycleAK}, Alon and Kozma obtained intriguing  identities -- involving the irreducible representations of the symmetric group -- for the expected number of cycles of a given size in $\xi_t$ on any finite graph.  Using these identities, they obtained a comparison-based estimate on the quantity
\begin{eqnarray}
t_{\textsc{cyc}}^{\IP}(G) & := & \min\left\{t\ge 0\colon \PP\left(\xi_t \textrm{ contains a     cycle of length}\ge \frac{|V_G|}{2}\right)\ge \frac 14\right\}.
\end{eqnarray}
\begin{theorem}[Alon and Kozma \cite{AK}]All finite graphs $G$ satisfy 
\begin{eqnarray}
\label{AK:cyc}
t_{\textsc{cyc}}^{\IP}(G) & \lesssim & \frac{\comp^{\IP}_G}{|V_G|}.
\end{eqnarray}
\end{theorem}
Thus, our main result implies that on high-dimensional graphs, macroscopic cycles do indeed emerge much before a single particle even mixes (recall (\ref{regime})). 
\begin{corollary}[Giant cycles]\label{co:cycles}All connected product graphs $G$ of fixed side-length $\ell$ satisfy
\begin{eqnarray}
t_{\textsc{cyc}}^{\IP}(G) & \lesssim_\ell &  1.
 \end{eqnarray} 
\end{corollary}
We note that Corollary \ref{co:cycles} may not be sharp: it is actually quite possible that the macroscopic cycles already emerge at time $\Theta(1/n)$ (where $n$ is the number of terms in the product), although proving this would require new ideas beyond the Alon-Kozma estimate (\ref{AK:cyc}). When specialized to the  hypercube $\Z_2^n$, Corollary \ref{co:cycles} complements a result of Koteck{\`y},  Mi{\l}o{\'s} and Ueltschi \cite{kotecky2016} regarding the appearance of mesoscopic cycles. It also complements a recent result by Adamczak, Kotowski and Mi{\l}o{\'s} \cite{adamczak2018phase}, who established a phase transition for the emergence of macroscopic cycles on the $2-$dimensional Hamming graph $\cK_n^2$. Finally, we note that, by virtue of \cite[Theorem 13]{AK}, our main result also has direct implications on the magnetisation of the quantum Heisenberg ferromagnet. 

\subsection{Exclusion process}
Another widely-studied interacting particle system is the {exclusion process}   \cite{MR3075635,Liggettbook1,Liggettbook2,MR1707314}. For a finite graph $G$ and an integer $0<k<|V_G|$, the \emph{$k-$particle exclusion process} (\EX)  on $G$ is a  Markov chain on the set ${V_G\choose k}$ of $k-$element subsets of $V_G$, with generator given by
\begin{eqnarray}
\cL_G^{\EX}f(S) & := & \sum_{e\in \partial S}\left(f(S\oplus e)-f(S)\right),
\end{eqnarray}
where $\oplus$ denotes the symmetric difference and $\partial S$ the   edge-boundary of $S$ in $G$. This process describes the set occupied by $k$ fixed particles under the \IP. More precisely, the \EX\ $(\zeta_t)_{t\ge 0}$ with initial condition $S\in {V_G\choose k}$ can be constructed from the \IP\  $(\xi_t)_{t\ge 0}$   by setting $\zeta_t:=\xi_t^{-1}(S)$. This observation, together with (\ref{Caputo}), easily implies that 
\begin{eqnarray}
\frac{\trel^{\RW}(G)}{\trel^{\RW}(\cK)}\ \le & \comp^{\EX}_G & \le \ \comp^{\IP}_G,
\end{eqnarray}
where $\cK$ denotes the complete graph on $V_G$, and $\comp^{\EX}_G$ the optimal constant in the functional inequality $\cE^{\EX}_{\cK}\left(\cdot\right)  \le  \comp_G^{\EX}\,\cE^{\EX}_G\left(\cdot\right)$. Recalling (\ref{0regime}) and the fact that $\trel^\RW(\cK)=1/|V_G|$, we obtain the following corollary.
\begin{corollary}[Comparison constant for \EX]\label{co:EX}For all connected product graphs $G$ of side-length $\ell\ge 2$, and all $0<k<|V_G|$, we have $\comp^{\EX}_G  \asymp_\ell   |V_G|.$
\end{corollary}
As a consequence, one can transfer many quantitative estimates from $\cK$ to $G$. This includes  the \emph{inverse log-Sobolev constant}  $\rho^{\EX}_{G}$, defined as the smallest number such that 
\begin{eqnarray}
\label{LSI}
\EE[f\log f]-\EE[f]\log\EE[f] & \le &  \rho^{\EX}_{G}\,\cE^{\EX}_G \left(\sqrt{f}\right)
\end{eqnarray}
 for all $f\colon  {V_G\choose k} \to (0,\infty)$, where $\EE[\cdot]$ is expectation under the uniform law. This constant provides powerful controls on the underlying Markov semi-group  \cite{MR1410112}.  It is easy to see that
\begin{eqnarray}
\frac 12\,\trel^{\RW}(G) \ \le & \rho^{\EX}_{G} & \le \ \comp_G^{\EX} \rho^{\EX}_{\cK}.
\end{eqnarray}
On the other hand, the log-Sobolev constant of the exclusion process on the complete graph (\emph{Bernoulli-Laplace model})  was determined by Lee and Yau \cite[Theorem 5]{LY}:
\begin{eqnarray}
\rho^{\EX}_{\cK_n} & \asymp & \frac{1}{n}\log \frac{n^2}{k(n-k)}.
\end{eqnarray} 
In particular, this allows us to pinpoint the exact order of  $\rho^{\EX}_G$ in the dense-particle regime.
\begin{corollary}[Log-Sobolev constant of \EX]Fix $\varepsilon\in(0,1)$, $\ell\ge 2$. Then, for all connected product graphs $G$ of side-length $\ell$ and all $k\in\left[\varepsilon|V_G|,(1-\varepsilon){|V_G|}\right]$, we have
$
\rho^{\EX}_G  \asymp_{\ell,\varepsilon}   1.
$
\end{corollary}
Finally, we note that our main result also implies an upper bound of order $n$ (uniformly in $1\le k\le 2^n$) on the $L^2$ mixing time of \EX\ on the hypercube $\Z_2^n$, complementing a total-variation estimate recently obtained by Hermon and  \cite{Pymar} (as part of a much more general result).

\section{Proof of the main result}
\subsection{Proof outline}
The lower bound in Theorem \ref{th:main} is easy. Indeed, if $G$ is any finite graph and if $\cK$ denotes the complete graph on $V_G$, then the very definition of $\chi_G^\IP$ implies
\begin{eqnarray}
\chi_G^\IP & \ge & \frac{\trel^\IP(G)}{\trel^{\IP}(\cK)} \ = \ \frac{\trel^\RW(G)}{\trel^{\RW}(\cK)} \ = \ |V_G|\,\trel^\RW(G),
\end{eqnarray}
where the first equality uses (\ref{Caputo}). For a graph product of side-length $\ell$, we deduce
\begin{eqnarray}
\chi_G^\IP & \gtrsim_\ell & |V_G|.
\end{eqnarray}
The remainder of the paper is devoted to proving a matching upper bound. To do so, we combine four simple ideas, each one corresponding to a step of the proof. 
\begin{enumerate}
\item Our first step consists in reducing the analysis of \IP\ on a   general $n-$dimensional graph-product $G$ of side-length $\ell$ to the special case of the Hamming graph $\cK_\ell^n$. This reduction relies on the classical method of canonical paths. An important simplification is that, by a standard path-lifting procedure, it is actually enough to just compare the single-particle  on $G$ to that on $\cK_\ell^n$. See  Section \ref{sec:cano} for details.
\item Our second step consists in re-interpreting the single-particle dynamics on $\cK_\ell^n$ as a random walk on the additive group $\Z_n^\ell$, with the increment law $\mu$ being uniform over vectors with a single non-zero coordinate. This algebraic reformulation is performed in Section \ref{sec:octopus}. It will allow one to use group-theoretical methods. 
\item The third step consists in exploiting the celebrated \emph{octopus inequality} \cite[Theorem 2.3]{Caputo} to compare the \IP\ with increment law $\mu$ to the \IP\ with increment law $\mu^{\star t}=\mu\star\cdots\star\mu$ ($t-$fold convolution), at a cost of order $t$. This is directly inspired by what Alon and Kozma did in \cite{AK}. However, instead of taking $t=\Theta(n\log n)$ so as to ensure that $\mu^{\star t}$ is close to uniform (all coordinates being refreshed with high probability), we crucially take $t=\Theta(n)$ only, with the prefactor being carefully adjusted so that only roughly \underline{half} of the coordinates get refreshed under $\mu^{\star t}$. This important point is made rigorous by an application of the de Moivre - Laplace Local Limit Theorem, see Section \ref{sec:LLT}.
\item Finally, the last step consists in showing that, although the increment law $\mu^{\star t}$ is still very far from uniform (because of our choice of $t$), the associated Dirichlet form is actually comparable to the one with uniform increments. This is achieved by constructing canonical paths of length $2$, the underlying intuition being that randomizing all coordinates of a vector can be achieved by randomizing half the coordinates in one step, and the other half in a second step. This is described in Section \ref{sec:final}.
\end{enumerate}
\subsection{Canonical paths}
\label{sec:cano}
Our starting point is a powerful  tool for comparing Dirichlet forms known as \emph{canonical paths}, see e.g., \cite{comparison}. As a warm-up, consider the single-particle dynamics (\ref{def:RW}) with Dirichlet form
\begin{eqnarray}
\cE^\RW_G(f) & := & \frac{1}{2|V_G|}\sum_{\{x,y\}\in E_G}\left(f(x)-f(y)\right)^2.
\end{eqnarray}
As usual, a \emph{path}  in $G$ will be a finite sequence of vertices $\gamma=(\gamma_0,\ldots, \gamma_k)$ $(k\ge 0)$ such that $e_i:=\{\gamma_{i-1},\gamma_i\}\in E_G$ for each $1\le i\le k$. We call $k$ the \emph{length} of the path and denote it by $|\gamma|$. Also, we refer to $\gamma_0,\gamma_k$ as the \emph{endpoints} of the path $\gamma$, and to $e_1,\ldots,e_k$ as the \emph{traversed} edges. By a \emph{random path} in $G$, we simply mean a random variable taking value in the (countable) set of all paths in $G$. We write $\EE[\cdot]$ for the corresponding expectation.
\begin{theorem}[Canonical paths, see e.g. \cite{comparison}]\label{th:cano}
Let $G$, $H$ be connected graphs on the same vertex set. For each edge $f\in E_H$, let $\gamma_f$ be any random path in $G$ with the same endpoints as $f$. Then, $\cE_{H}^{\RW}  \le  \kappa\,\cE_{G}^{\RW}$ where  $\kappa$ is the congestion, defined as follows:
\begin{eqnarray}
\label{def:congestion}
\kappa & = & \max_{e\in E_G}\left\{\sum_{f\in E_H}\EE\left[|\gamma_{f}|{\bf 1}_{\left(\gamma_{f}\textrm{ traverses }e\right)}\right]\right\}.
\end{eqnarray}
\end{theorem}
We now make three elementary but important remarks. 
\begin{remark}[Trivial choice]\label{rk:crude}Even in the worst-case situation where $H$ is the complete graph on $V_G$, we can always achieve the  poor  bound
\begin{eqnarray}
\label{worst}
\kappa & \le & \frac{|V_G|^3}{4},
\end{eqnarray}
by considering a spanning tree $T$ of $G$ and letting $\gamma_f$ be the unique simple path in $T$ connecting the endpoints of $f$. Note that this path is actually non-random. Exploiting randomness and the particular structure of $H$ to design paths with a low congestion is a matter of art. 
\end{remark}
\begin{remark}[Congestion behaves well under products]\label{rk:prod}
If for $1\le i\le n$, we can compare $G_i$ to $H_i$ with congestion $\kappa_i$, then we can compare $G_1\times \cdots\times G_n$ to $H_1\times\cdots\times H_n$ with congestion 
\begin{eqnarray}
\kappa & = & \max\left(\kappa_1,\ldots,\kappa_n\right),
\end{eqnarray}
by considering paths that only evolve along a single coordinate, in the obvious way.
\end{remark}
\begin{remark}[Cayley graphs]\label{rk:caylay}Theorem \ref{th:cano} simplifies when $G=\textrm{Cay}(\GG,A)$ and $H=\textrm{Cay}(\GG,B)$ are Cayley graphs generated by subsets $A,B$ of a finite group $\GG$. Indeed, any word $\omega=(\omega_1,\ldots,\omega_k)\in A^k$  can be used to define a path 
\begin{eqnarray}
\gamma_{x,xb} & := & \left(x,x\omega_1,\ldots,x\omega_1\cdots\omega_k\right)
\end{eqnarray}
in $G$ from $x\in\GG$ to  $xb$, where $b=\omega_1\cdots\omega_k$ is the evaluation of $\omega$ in $\GG$. Consequently, we only have to specify, for each $b\in B$, a random word $\omega_b$ over $A$ whose evaluation is $b$. Moreover, a straightforward computation shows that the resulting congestion is simply
\begin{eqnarray}
\kappa & = & \max_{a\in A}\left\{\sum_{b\in B}\EE\left[|\omega_b|N(a,\omega_b)\right]\right\},
\end{eqnarray}
where  $|\omega|$ denotes the length of a word $\omega$, and $N(a,\omega)$ the number of occurrences of $a$ in it. 
\end{remark}
  
Remark \ref{rk:caylay} applies in particular to the \IP\ on any graph $G$. Indeed, one has
\begin{eqnarray}
\cE_G^{\IP} & = & \cE^{\RW}_{\cay(\GG,A)},
\end{eqnarray}
with $\GG=\fS(V_G)$ and $A=\{\tau_e\colon e\in E_G\}$. Moreover, any path in $G$ with endpoints $f=\{x,y\}$ and traversed edges $e_1,\ldots,e_k$  can be \emph{lifted} to a word over  $A$ that evaluates to $\tau_{f}$, namely:
\begin{eqnarray}
\omega & := &  \left(\tau_{e_1},\cdots,\tau_{e_{k-1}},\tau_{e_k },\tau_{e_{k-1}},\ldots,\tau_{e_1}\right).
\end{eqnarray}
Since the congestion is multiplied by at most $4$ ($2$ for the length of the word, and $2$ for the number of occurrences of a letter in it), we obtain the following classical result.
\begin{corollary}[From canonical paths for $\RW$ to canonical paths for \IP]\label{co:cano}Under the exact same assumptions (and notation)  as in Theorem \ref{th:cano}, we also have
\begin{eqnarray}
\cE_{H}^{\IP} & \le & 4\kappa\,\cE_{G}^{\IP}.
\end{eqnarray}
\end{corollary}
Combining this  with Remarks \ref{rk:crude} and \ref{rk:prod}, we obtain the following inequality, which reduces   the upper bound of Theorem \ref{th:main} to the extremal case where $G$ is the $(n,\ell)-$\emph{Hamming graph}: 
\begin{eqnarray}
\cK_{\ell}^n & := & \underbrace{\cK_\ell\times\cdots\times \cK_\ell}_{n\text{ times}}.
\end{eqnarray} 
\begin{corollary}\label{lm:reduc} For any $n-$dimensional connected product graph $G$ of side-length $\ell$, 
 \begin{eqnarray}
\cE^{\IP}_{\cK_{\ell}^n} & \le & \ell^3\,\cE^{\IP}_{G}.
 \end{eqnarray}
\end{corollary}
In light of this result, the upper bound in Theorem \ref{th:main} now boils down to the claim
\begin{eqnarray}
\sup_{n\ge 1}\left\{\frac{\comp^{\IP}_{\cK_{\ell}^n}}{\ell^{n}}\right\}& < &  \infty,
\end{eqnarray}
for each $\ell\ge 2$, to which the remainder of the paper is devoted. 
\subsection{The octopus inequality}\label{sec:octopus}
From now on, we fix the side-length $\ell\ge 2$ and the  dimension $n\ge 1$. Writing $\cK$ for the complete graph on $\{1,\ldots,\ell\}^n$, our goal is to establish the comparison
\begin{eqnarray}
 \label{task}
\cE^\IP_{\cK} & \le & c\,\ell^n\, \cE^\IP_{\cK_{\ell}^n},
\end{eqnarray}
where $c$  does not depend on $n$. We start by observing that the random walks on $\cK_{\ell}^n$  and on $\cK$ can both be conveniently viewed as random walks on the group 
\begin{eqnarray}
\GG & := & \Z_\ell^n,
\end{eqnarray}
equipped with coordinate-wise addition mod $\ell$ (which we will simply denote by $+$). Given a probability measure $\mu$ on $\GG$, we recall that the random walk with increment law  $\mu$ has Dirichlet form  
\begin{eqnarray}
\cE^\RW_\mu(f) & := & \frac{1}{2|\GG|}\sum_{x,z\in\GG}\mu(z)\left(f(x+z)-f(x)\right)^2,
\end{eqnarray}
for all $f\colon \GG\to\R$.  In particular,  we have the representation
\begin{eqnarray}
\label{representation}
\cE^\RW_\cK \  = \ (\ell^n-1)\cE_\pi^\RW, & \qquad & \cE^\RW_{\cK_{\ell}^n} \ = \ n(\ell-1)\cE_{\rho_1}^\RW,
\end{eqnarray}
where $\pi$ and $\rho_k$ $(0\le k\le n)$ respectively denote the uniform distributions on $\GG$ and on
\begin{eqnarray}
\GG_k & := & \left\{x\in\GG\colon |\supp(x)|=k\right\}.
\end{eqnarray}
Here $\supp(x)=\{i\colon x_i\ne 0\}$ naturally denotes the support of $x=(x_1,\ldots,x_n)\in \GG$.  
Similarly,   the  \IP\ on $\GG$ with increment law $\mu$ has Dirichlet form  
\begin{eqnarray}
\cE^\IP_\mu(f) & := & \frac{1}{2|\GG|!}\sum_{\sigma\in\fS(\GG)}\sum_{x,z\in\GG}\mu(z)\left(f(\sigma\tau_{\{x,x+z\}})-f(\sigma)\right)^2,
\end{eqnarray}
for $f\colon \fS(\GG)\to\R$, with the interpretation $\tau_{\{x,x+z\}}=\textrm{id}$ when $z=0$. In view of (\ref{representation}) (with $\IP$ instead of $\RW$), our claim (\ref{task}) rewrites as
\begin{eqnarray}
\label{reform:1}
\cE^\IP_{\pi} & \le & c \, n\, \cE^\IP_{\rho_1},
\end{eqnarray}
for some (possibly different) constant $c<\infty$ that is only allowed to depend on $\ell$.
The proof  will crucially rely on the following elegant application of the octopus inequality \cite[Theorem 2.3]{Caputo}, which  we borrow from Alon and Kozma \cite{AK}. We include a short proof as our setting is here slightly different. 
The convolution of two probability measures  $\mu,\nu$   on $\GG$ is defined by
\begin{eqnarray}
(\mu\star\nu)(x) & := & \sum_{z\in\GG}\mu(z)\nu(x-z).
\end{eqnarray}
Also, we say that a measure $\mu$ on $\GG$ is symmetric if $\mu(z)=\mu(-z)$ for all $z\in\GG$.

\begin{lemma}[Comparison for convolutions]\label{lm:convo}For any symmetric probability measure $\mu$ on $\GG$, 
\begin{eqnarray}
\cE_{\mu\star\mu}^\IP & \le & 2\cE_\mu^{\IP}.
\end{eqnarray}
\end{lemma}
\begin{proof}If $\mu(0)=0$, the octopus inequality \cite[Theorem 2.3]{Caputo} asserts that  
\begin{eqnarray*}
\sum_{\sigma\in\fS(\GG),z\in\GG}\mu(z)\left(f(\sigma\tau_{\{x,x+z\}})-f(\sigma)\right)^2 & \ge & \frac{1}{2}\sum_{\sigma\in\fS(\GG),(u,v)\in\GG^2}\mu(u)\mu(v)\left(f(\sigma\tau_{\{x+u,x+v\}})-f(\sigma)\right)^2,
\end{eqnarray*}
for all $f\colon\fS(\GG)\to\R$ and  $x\in\GG$ (the factor $\frac 12$ on the right-hand side compensates for the fact that we are here summing over all \underline{ordered} pairs $(u,v)\in\GG^2$). Averaging over all $x\in\GG$,  and applying the (bijective) change of variables $(x,u,v)\mapsto(x+u,-u,v-u)$ on the right-hand side, we obtain
\begin{eqnarray}
2\cE_\mu^\IP(f) & \ge &  \frac{1}{2|\GG|!}\sum_{\sigma\in\fS(\GG),(x,u,v)\in\GG^3}\mu(-u)\mu(v-u)\left(f(\sigma\tau_{\{x,x+v\}})-f(\sigma)\right)^2,
\end{eqnarray}
which is precisely $2\cE_\mu^\IP(f)\ge \cE_{\mu\star\mu}^\IP(f)$ by symmetry of $\mu$. This proves the claim when $\mu(0)=0$. In the general case, we write $\mu=(1-\theta)\rho_0+\theta\nu$ with ${\nu}(0)=0$, and we observe that
\begin{eqnarray*}
\cE_\mu^\IP \ = \  \theta\cE_\nu^\IP, & \qquad & 
\cE_{\mu\star\mu}^\IP \ = \ \theta^2\cE_{\nu\star\nu}^\IP+2\theta(1-\theta)\cE_{\nu}^\IP.
\end{eqnarray*}
Thus, the claim $\cE_{\mu\star\mu}^\IP\le 2\cE_\mu^\IP$ is equivalent to $\cE_{\nu\star\nu}^\IP\le 2\cE_\nu^\IP$.
\end{proof}
For reasons that will become clear later, we henceforth set
\begin{eqnarray}
\label{def:t}t & := & 2^{\lceil \log_2 n\rceil} \ \in \ [n,2n]\\
\label{def:theta}\theta & := & n\left(1-e^{-\frac{\ln 2}{t}}\right) \ \in \ \left(0,1\right)\\
\label{def:p}p & := & \frac{\ell-1}{\ell}.
\end{eqnarray}
Let us introduce the measure $\mu$ defined by
\begin{eqnarray}
\mu  & := & (1-\theta p)\rho_0+\theta p\rho_1.
\end{eqnarray}
Since $\mu$ is symmetric and $t$ is a power of $2$, we may iterate Lemma \ref{lm:convo} to get
\begin{eqnarray}
\cE_{\mu^{\star t}}^\IP & \le & t\cE_{\mu}^\IP \ = \ \theta pt\cE_{\rho_1}^\IP \ \le \ 2n\cE_{\rho_1}^{\IP}.
\end{eqnarray}
where $\mu^{\star t}=\mu\star\cdots\star\mu$ denotes the $t-$fold convolution of $\mu$. 
Thus, our goal (\ref{reform:1}) now boils down to showing that 
\begin{eqnarray}
\label{strategy}
\cE_\pi^\IP & \le & 
c\, \cE_{\mu^{\star t}}^\IP
\end{eqnarray}
for some constant $c<\infty$ that only depends  on $\ell$. To this end, we   analyze the convolution $\mu^{\star t}$ accurately using the de Moivre - Laplace Local Limit Theorem.
\subsection{Local Limit Theorem}
\label{sec:LLT}
As a warm-up, consider the binomial expansion of the uniform law: $\pi  =  \sum_{k=0}^n b_k\rho_k$, where 
\begin{eqnarray}
b_k & := & {n \choose k}\left(1-p\right)^kp^{n-k}.
\end{eqnarray}
The classical de Moivre - Laplace Local Limit Theorem provides uniform estimates on the coefficients $b_0,\ldots,b_n$. Although a specific value of $p$ was chosen at (\ref{def:p}), the statement is of course valid for any $p\in(0,1)$.
\begin{theorem}[de Moivre - Laplace]\label{th:LLT}There is $C<\infty$ depending only on $p$ such that 
\begin{eqnarray}
\left|b_k - \frac{e^{-\frac{x^2}{2}}}{\sqrt{2\pi np(1-p)}}\right| & \le & \frac{C}{n^{3/2}},
\end{eqnarray}
for all $0\le k\le n$, with $x=(k-np)/\sqrt{np(1-p)}$.
\end{theorem}
We can use this to approximate $\cE_{\pi}$ with $\cE_{\rho_\cI}$, where $\rho_\cI$ is defined as follows:
\begin{eqnarray}
\rho_{\cI} & := & \frac{1}{|\cI|}\sum_{k\in \cI}\rho_k;\\
\cI & := & \left(np-2\sqrt{np(1-p)},np+2\sqrt{np(1-p)}\right)\cap\{0,\ldots,n\}.
\end{eqnarray}
\begin{lemma}[Plateau proxy for $\cE_{\pi}^\IP$]\label{lm:plateaupi}
There is $c<\infty$ depending on $\ell$ only, such that
\begin{eqnarray}
\cE_{\pi}^\IP & \le & c\,\cE_{\rho_\cI}^\IP.
\end{eqnarray}
\end{lemma}
\begin{proof} If $\nu$ is a symmetric distribution on $\GG$, the Cauchy-Schwartz inequality  yields
\begin{eqnarray*}
\sqrt{\frac{(\nu\star\nu)(x)}{\pi(x)}}
& \ge & \sum_{z\in\GG}\sqrt{\nu(z)\nu(z-x)}\\
& \ge & \sum_{z\in\GG}\nu(z)\wedge\nu(z-x)\\
& \ge & \sum_{z\in\GG}\left[\nu(z)-\left(\nu(z)-\frac{1}{|\GG|}\right)_+-\left(\frac{1}{|\GG|}-\nu(z-x)\right)_+\right]\\
& = & 1-2\dtv(\nu,\pi),
\end{eqnarray*}
for all $x\in\GG$, where $\dtv(\cdot,\cdot)$ denotes the total-variation distance. In particular, when $\dtv(\nu,\pi)\le \frac{1}{4}$,  we obtain 
$(\nu\star\nu)(x) \ge \pi(x)/4$ for all $x\in\GG$ and hence
\begin{eqnarray}
\label{lm:tv}
\cE_{\pi}^\IP & \le & 4\cE_{\nu\star\nu}^\IP.
\end{eqnarray}
Let us now apply this general observation to the restriction of $\pi$ to $\bigcup_{k\in \cI}\GG_k$:
\begin{eqnarray}
\nu \ := \ \frac{1}{q}\sum_{k\in \cI} b_k\rho_k, & \qquad & q\ :=\ \sum_{k\in \cI}b_k.
\end{eqnarray}
Note that $\dtv(\nu,\pi)=1-q$, and that  $q\ge 3/4$ thanks to our definition of $\cI$ and Chebychev's inequality for the Binomial$(n,p)$. Thus, (\ref{lm:tv}) applies and yields
\begin{eqnarray*}
\cE_{\pi}^\IP & \le & 4\cE_{\nu\star\nu}^\IP\\
& \le & 8\cE_\nu^\IP\\ & = & \frac{8}{q}\sum_{k\in\cI}b_k\cE_{\rho_k}^\IP\\ 
& \le & \frac{32|\cI|}{3}\left(\max_{k\in\cI}b_k\right)\cE_{\rho_\cI}^\IP,
\end{eqnarray*}
where the second inequality uses Lemma \ref{lm:convo}, and the third $q\ge 3/4$. Finally, Theorem \ref{th:LLT} ensures that $|\cI|\max_{k\in \cI}b_k$ is bounded by a quantity which only depends on $p$.
\end{proof}

 In order to establish (\ref{strategy}), we will now approximate $\mu^{\star t}$ by the distribution
\begin{eqnarray}
\rho_{\cJ} & := & \frac{1}{|\cJ|}\sum_{k\in \cJ}\rho_k;\\
\cJ & := & \left(\frac{np}2-2\sqrt{np(1-p)},\frac{np}2+2\sqrt{np(1-p)}\right)\cap\{0,\ldots,n\}.
\end{eqnarray}
Note that the center of $\cJ$ is twice smaller than that of $\cI$. 
\begin{lemma}[Plateau proxy for $\mu^{\star t}$]\label{lm:plateaumu}There is $c>0$ depending on $\ell$ only, such that
\begin{eqnarray}
\mu^{\star t} & \ge & c\,  \rho_\cJ.
\end{eqnarray}
\end{lemma}
\begin{proof}
The convolution with $\mu$ describes the following transformation on $\GG-$valued random variables:  pick one of the $n$ coordinates uniformly at random and, with probability $\theta$, replace it with a fresh uniform sample from $\Z_\ell$. Consequently, we may construct a random variable $X=(X_1,\ldots,X_n)$ with law $\mu^{\star t}$   by setting
\begin{eqnarray}
X_i & := & \left\{ 
\begin{array}{ll}
Z_i & \textrm{if }i\in \{U_1,\ldots,U_N\}\\
0 & \textrm{otherwise},
\end{array}
\right.
\end{eqnarray}
where $N,U_1,\ldots,U_t,Z_1,\ldots,Z_n$ are independent random variables with the following laws:
\begin{itemize}
\item $N$ is binomial with parameters $t$ and $\theta$;
\item $U_1,\ldots,U_t$ are uniform  on $\{1,\ldots,n\}$;
\item $Z_1,\ldots,Z_n$ are uniform  on $\Z_\ell$.
\end{itemize}
In particular, setting $S:=|\supp(X)|$, we have
\begin{eqnarray}
\label{def:ak}
\mu^{\star t} & = & \sum_{k=0}^n \PP\left(S=k\right) \rho_k,
\end{eqnarray}
and our proof boils down to establishing that
\begin{eqnarray} 
\label{tt}
\min_{k\in\cJ}\PP\left(S=k\right) & \ge & \frac{c}{\sqrt{n}},
\end{eqnarray}
for some constant $c>0$ that only depends on $\ell$. Now, conditionally on $N$, the  variable 
\begin{eqnarray}
R & := & |\{U_1,\ldots,U_N\}|
\end{eqnarray} counts the number of distinct coupons collected by time $N$ in the standard coupon-collector problem of size $n$. Thus,
\begin{eqnarray}
 \EE\left[R|N\right] \ = \ n\left(1-\left(1-\frac{1}{n}\right)^N\right), & \qquad & 
 \textrm{Var}(R|N) \ \le \ \frac{n}{4}.
 \end{eqnarray} 
Recalling our definitions (\ref{def:theta}), and since $N$ is a Binomial$(t,\theta)$ variable, we easily deduce 
\begin{eqnarray}
\EE\left[R\right] \ = \ n\left(1-\left(1-\frac{\theta}{n}\right)^t\right) \ = \ \frac{n}{2}
& \qquad  &
\textrm{Var}(R) \ \le \ n.
\end{eqnarray}
Consequently, Chebychev's inequality yields
\begin{eqnarray}
\label{cheby1}
\PP\left(R\in\left[\frac{n}2-2\sqrt{n},\frac{n}2+2\sqrt{n}\right]\right) &  \ge & \frac{3}{4}.
\end{eqnarray}
Now,  conditionally on $R$, the random variable $S$ is just a Binomial with parameters $R,p$. In particular, Theorem \ref{th:main} with $R$ instead of $n$ ensures that
\begin{eqnarray}
\label{cheby2}
\min_{k\in\cJ}\PP\left(S=k \left|R\in\left[\frac{n}{2}-2\sqrt{n},\frac{n}{2}+2\sqrt{n}\right]\right.\right) & \ge & \frac{c}{\sqrt{n}},
\end{eqnarray}
where $c>0$ only depends on $p$.  Combining (\ref{cheby1}) and (\ref{cheby2}) establishes the claim.
\end{proof}

\subsection{Final comparison}
\label{sec:final}
In view of Lemmas \ref{lm:plateaupi} and \ref{lm:plateaumu}, our objective (\ref{strategy}) now reduces to establishing the following. 
\begin{proposition}[Final comparison]\label{pr:final}There exists $c<\infty$ depending only on $\ell$, such that
\begin{eqnarray}
\cE^{\IP}_{\rho_\cI} & \le & c\,\cE^{\IP}_{\rho_\cJ}.
\end{eqnarray}
\end{proposition}
The crucial ingredient of the proof is the following lemma.
\begin{lemma}\label{lm:final}For any $i,j\in\{0,\ldots,n\}$ with $i+j\in\{0,\ldots,n\}$, we have
\begin{eqnarray}
\cE_{\rho_{i+j}}^\IP & \le & \frac{8|\GG_i|}{|\GG_i|\wedge|\GG_j|}\cE_{\rho_i}^\IP+\frac{8|\GG_j|}{|\GG_i|\wedge|\GG_j|}\cE_{\rho_{j}}^\IP.
\end{eqnarray}
\end{lemma} 
\begin{proof}
Let $(X,Y)$ denote a random element from the set 
\begin{eqnarray}
\left\{(x,y)\in\GG_i\times \GG_j\colon \supp(x)\cap\supp(y)=\emptyset\right\}.
\end{eqnarray}
Then $\omega:=(X,Y)$ is a random word of length $2$ over $\GG_i\cup\GG_j$, whose evaluation $X+Y$ is uniform over $\GG_{i+j}$. By Corollary \ref{co:cano} and Remark \ref{rk:caylay}, we deduce that
\begin{eqnarray}
\cE_{\cay(\GG,\GG_{i+j})}^\IP & \le & 4\kappa\,\cE_{\cay(\GG,\GG_i\cup\GG_j)}^\IP
\end{eqnarray}
where the congestion $\kappa$ is given by
\begin{eqnarray}
\kappa & = & 2|\GG_{i+j}|\max_{b\in \GG_i\cup\GG_j}\left\{\PP\left(X=b\right)+\PP\left(Y=b\right)\right\}\\ & = & \frac{2|\GG_{i+j}|}{|\GG_i|\wedge|\GG_j|}\left(1+{\bf 1}_{(i=j)}\right).
\end{eqnarray}
The second line follows from the observation that $X$ and $Y$ are uniformly distributed on $\GG_i$ and $\GG_j$, respectively. On the other hand, the definitions of $\rho_i,\rho_j,\rho_{i+j}$ imply
\begin{eqnarray}
\cE_{\cay(\GG,\GG_{i+j})}^\IP & = & |\GG_{i+j}|\,\cE_{\rho_{i+j}}^\IP\\
\left(1+{\bf 1}_{(i=j)}\right)\cE_{\cay(\GG,\GG_{i}\cup\GG_j)}^\IP & = &  |\GG_i|\cE_{\rho_i}^\IP+|\GG_j|\cE_{\rho_j}^\IP.
\end{eqnarray}
The claim readily follows. 
\end{proof}
\begin{proof}[Proof of Proposition \ref{pr:final}]
Our definitions of $\cI,\cJ$ ensure that $|\cI|=|\cJ|$ and that
\begin{eqnarray}
\cI  \subseteq  \left\{i+j\colon (i,j)\in\cP\right\}, & \textrm{ where } & \cP := \left\{(i,j)\in\cJ^2\colon j\in\{i,i+1\}\right\}.
\end{eqnarray}
In particular, we have
\begin{eqnarray}
\cE_{\rho_{\cI}}^\IP 
& \le & \frac{1}{|\cJ|}\sum_{(i,j)\in\cP}\cE_{\rho_{i+j}}^\IP.
\end{eqnarray}
Now, since $|\GG_i|={n \choose i}(\ell-1)^i$ for all $i\in\{0,\ldots,n\}$, we have 
\begin{eqnarray}
\frac{|\GG_{i+1}|}{|\GG_{i}|} & = & \frac{(\ell-1)(n-i)}{(i+1)}.
\end{eqnarray}
As $i$ varies across $\cJ$, this ratio remains bounded away from $0$ and $\infty$ uniformly in $n$. Consequently, Lemma \ref{lm:final} ensures that for all $(i,j)\in\cP$,
\begin{eqnarray}
\cE_{\rho_{i+j}}^\IP & \le & c\cE_{\rho_i}^\IP+c\cE_{\rho_j}^\IP,
\end{eqnarray}
where $c<\infty$ depends only on $\ell$. Inserting this above, we obtain 
\begin{eqnarray}
\cE_{\rho_{\cI}}^\IP & \le & \frac{c}{|\cJ|}\sum_{(i,j)\in\cP}(\cE_{\rho_{i}}^\IP+\cE_{\rho_j}^\IP)\\
& \le &  \frac{4c}{|\cJ|}\sum_{j\in\cJ}\cE_{\rho_j}^\IP \ = \ 4c\cE_{\rho_{\cJ}}^\IP.
\end{eqnarray}
This concludes the proof.
\end{proof}

\bibliographystyle{plain}
\bibliography{IPLS}

\end{document}